\documentclass[11pt]{amsart}

\usepackage[all,cmtip]{xy}
\usepackage{amsmath, amssymb}
\usepackage{mathtools}
\usepackage[T1]{fontenc}
\usepackage[english]{babel}
\usepackage[hidelinks]{hyperref}
\usepackage[marginpar=2.5cm]{geometry}
\usepackage{xcolor}\usepackage{marginnote}
\usepackage{amsrefs}

\newtheorem{thm}{Theorem}

\newtheorem*{thm*}{Theorem}
\newtheorem{lem}[thm]{Lemma}

\newtheorem*{prop*}{Proposition}

\newtheorem{cor}[thm]{Corollary}

\theoremstyle{definition}

\newtheorem{remark}[thm]{Remark}

\def\bb{\mathbb}

\def\vp{\varphi}

\def\bb{\mathbb}

\def\cc{\mathcal}

\newcommand\ip[2]{\left\langle #1\, , #2 \right\rangle}

\begin{document}

\title[A short proof of the existence of the injective envelope]{A short proof of the existence of the injective envelope of an operator space}

\author[T. Sinclair]{Thomas Sinclair}
\address{Department of Mathematics, Purdue University, 150 N University St, West Lafayette, IN 47907-2067, USA}
\email{tsincla@purdue.edu}


\dedicatory{}

\begin{abstract} We use Ellis' lemma to give a simple proof of the existence of the injective envelope of an operator space first shown by work of Hamana and Ruan.

\end{abstract}

\maketitle


Ramsey-theoretic methods provide powerful tools in functional analysis, ergodic theory, and additive combinatorics: see \cite{at, dgl, hs} for many of these applications. The goal of this note is to give a short, elementary proof of the existence of the injective envelope of an operator space via Ellis' lemma on right topological semigroups. The theory of injective envelopes for the categories of C$^\ast$-algebras and operator systems and their associated dynamical systems is developed in the seminal works of Hamana \cite{ha, ha1, ha2, ha3}, while Ruan first showed the existence of injective envelopes for operator spaces \cite{ru}.  We mention that the treatment of injective envelopes as in \cite{pa1, pa, pa3} is very close in spirit with our proof. In fact, in \cite{pa3} the injective envelope of certain crossed products of a countable, discrete $G$ is related with idempotents in the Stone--Cech compactification $\beta G$,  but it seems a formal connection with Ellis' lemma in the general context was never made.

Let $S$ be a nonempty Hausdorff topological space equipped with a semigroup operation so that $S\ni s\mapsto st$ is continuous for each $t\in S$ separately. We will say that $S$ is a \emph{right topological semigroup}. Let $\cc I(S)$ be the (possibly empty) set of idempotent elements of $S$ which we equip with the natural partial order $e\preceq f$ if $ef = fe = e$. A \emph{minimal idempotent} is an idempotent which is a minimal element of the poset $(\cc I(S), \preceq)$. Additionally we say that two idempotents $e,f\in \cc I(S)$ are \emph{similar}, $e\sim f$, if $e = fe$ and $f = ef$. We note that similarity is an equivalence relation on $\cc I(S)$.

We now state Ellis' lemma \cite{el}, a fundamental and far-reaching result on the structure of compact right topological semigroups.

\begin{lem}
    Every compact right topological semigroup contains an idempotent.
\end{lem}

In addition to Ellis' paper cited above, proofs can be found, for instance, as \cite[Theorem 1.23]{dgl}, \cite[Theorem 1.1]{fk}, or \cite[Theorem 2.5]{hs}.

The following consequence of Ellis' lemma essentially appears, for instance, as Theorems 1.2-1.4 in \cite{fk} or as a combination of Theorem 1.60, Theorem 2.5, and Corollary 2.6 in \cite{hs}.

\begin{lem}\label{ellis} If $S$ is a compact right topological semigroup, then for every idempotent $e$ there is a minimal idempotent $f\preceq e$.
\end{lem}
We offer a streamlined proof for the convenience of the reader.

\begin{proof}
    Let $e\in S$ be an idempotent, and consider the left ideal $Se$, which is closed by continuity of right multiplication. By compactness of $Se$ and Zorn's lemma, there is $J\subseteq Se$ a minimal closed left ideal. Noting that $J$ itself is a compact right topological semigroup, $J$ contains an idempotent $f$. 
    
    By minimality $J = Sx$ for all $x\in J$, thus $g = yh = yh^2 = gh$ for all $g,h\in J$ idempotent. If $g\in \cc I(S)$ and $g\preceq f$, then $g = gf$ implies $g\in J$. It  follows that $g = fg = f$, so $f$ is minimal.
    Since $f\in Se$, $f = fe$, and we have $(ef)^2 = efef = ef^2 =ef$ is an idempotent in $J$, hence minimal. As $ef = eef = efe$, $ef\preceq e$.
\end{proof}

\begin{remark}\label{remark}
    If $S$ is a semigroup, $e,f\in\cc I(S)$, and $fe =e$, then $ef\preceq f$. Consequently, if $f$ is minimal, then $fe =e$ implies that $e\sim f$.
\end{remark}

Let $S$ be a convex subset of a locally convex topological vector space $V$ which is equipped with a semigroup structure that is affine and continuous with respect to right multiplication. Following \cite{berglund}, we will refer to $S$ as a \emph{affine right topological semigroup}. 
 
The following result and its proof were communicated to the author by Matthew Kennedy. The result appears in essentially this form as \cite[Theorem II.4.3]{berglund}.

\begin{lem}\label{kennedy}
    Let $S$ be a compact affine right topological semigroup, and let $J$ be a minimal closed left ideal. We have that $J$ is a left zero semigroup, that is, $xy = x$ for all $x,y\in J$. In particular, every element of a minimal closed left ideal of $S$ is idempotent.
\end{lem}

\begin{proof}
    We have that $J = Sy$ for any $y\in J$, thus $J$ is compact and convex. We have that $x\mapsto xy$ is a continuous affine map from $J$ to itself, hence by the Markov--Kakutani fixed point theorem \cite[Theorem V.10.1]{conway} there is $x\in J$ so that $xy = x$. The set of all $x\in J$ so that $xy = x$ is thus nonempty, closed, and a left ideal, hence is equal to $J$ by minimality. 
\end{proof}

\begin{remark}
    A related result appears in work of Marrakchi \cite[Theorem 3.6]{marrakchi}, inspired by an earlier version of this manuscript. It states that if $S$ is a compact affine right topological semigroup and $e\in S$ is a minimal idempotent, then $exe = e$ for all $x\in S$. We now show that this result is equivalent to Lemma \ref{kennedy}.
    
    Since we have seen in the proof of Lemma \ref{ellis} that every minimal closed left ideal in a right topological semigroup is of the form $Sf$ for some minimal idempotent, Marrakchi's result shows that $xy = xfyf = xf = x$ for all $x,y\in Sf$.  Conversely, if $f$ is a minimal idempotent, then $Sf$ is a minimal closed left ideal. Indeed, if $J\subseteq Sf$ is a minimal closed left ideal, then for any (minimal) idempotent $g\in J$, $g = gf$, which shows that $fg$ is idempotent with $fg\preceq f$. Hence $fg = f$ by minimality, and it follows that $J\supseteq Sg = Sf$. By Lemma \ref{kennedy}, $fxf = f(xf) =f$.
\end{remark}

Let $X = Y^*$ be a dual Banach space. 
Denote $\cc B(X)$ to be the algebra of a bounded linear operators on $X$ and $\cc C(X)$ to be the affine subsemigroup of contractions.

\begin{lem}
    We have that $\cc C(X)$ is a compact affine right topological semigroup under operator composition and convergence in the pointwise-weak* topology.
\end{lem}

\begin{proof}
    We have that $f\mapsto f\circ g$ is clearly affine in $f$ for all $f,g\in\cc C(X)$ and further that if $f_n\to f$ in the pointwise-weak* topology, then $f_n\circ g\to f\circ g$ in the pointwise-weak* topology.

    Since $X = Y^*$ is a dual Banach space, $\cc B(X)$ is isometrically isomorphic to the dual of the projective tensor product $X\otimes_\pi Y$ via the correspondence $\ip{x\otimes y}{\vp} \leftrightarrow \ip{y}{T_\vp(x)}$, \cite[section 2.2]{ryan}. In this way the weak* topology on $\cc B(X)$ can be seen to coincide with the topology of pointwise-weak* convergence on functions from $X$ to itself. The Banach--Alaoglu theorem then applies, establishing compactness. \qedhere
\end{proof}

We now give the main applications to the theory of operator spaces. In the following, $H$ will denote a Hilbert space, and $\cc B(H)$ will be the algebra of all bounded linear operators on $H$. We recall that $\cc B(H)$ is a dual Banach space. It is then straightforward to check that $\cc{CC}(\cc B(H))$, the set of all completely contractive linear maps from $\cc B(H)$ to itself, it is a closed convex subsemigroup of $\cc C(\cc B(H))$ in the pointwise-weak* topology. Thus, $\cc{CC}(\cc B(H))$ is a compact affine right topological semigroup in its own right.

Let $E\subset \cc B(H)$ be an operator space. We say that $E$ is \emph{injective} in the category of operator spaces if for all inclusions $A\subset B$ of operator spaces and all completely contractive maps $\vp: A\to E$, there is a completely contractive extension of $\vp': B\to E$. Since $\cc B(H)$ is injective for any Hilbert space by Wittstock's extension theorem \cite[Theorem 8.2]{pa}, $E\subseteq \cc B(H)$ is injective if and only if there is a completely contractive idempotent $\vp: \cc B(H)\to \cc B(H)$ with $E = \vp(\cc B(H))$.

Following \cite[Chapter 15]{pa}, we will say that a pair $(F, \kappa)$ consisting of an operator space $F$ and a completely isometric embedding $\kappa: E\to F$ is an \emph{injective envelope} for $E$ if $F$ is an injective object in the category of operator spaces and for all injective operator spaces $\kappa(E)\subseteq F_0 \subseteq F$, $F_0 = F$.

\begin{thm}\label{thm:main} Any operator space has an injective envelope.
\end{thm}

\begin{proof} Let $E\subset \cc B(H)$ be an operator space. Let $\cc S$ be the set of all $\vp\in \cc{CC}(\cc B(H))$ so that $\vp(x) = x$ for all $x\in E$. It is easy to check that $\cc S$ is a closed convex subsemigroup  of $\cc{CC}(\cc B(H))$ in the topology of pointwise-weak* convergence. Therefore, by Lemma \ref{ellis} there is a minimal idempotent $\varphi\in \cc S$. Let $F= \vp(\cc B(H))$. Suppose that $E \subseteq F_0\subseteq F$ is injective, as witnessed by $\psi\in \cc I(\cc S)$. We have $\vp\circ\psi = \psi$, hence $\vp\sim \psi$ by Remark \ref{remark}. This implies that $\vp$ and $\psi$ have the same range as linear operators, thus $F_0 = F$. \qedhere
\end{proof}

\begin{remark}
    This theorem effectively characterizes the injective envelopes of $E\subseteq \cc B(H)$ as the images of minimal idempotents in the semigroup of completely contractive self-maps of $\cc B(H)$ which pointwise fix $E$.
\end{remark}

The next result was brought to the author's attention by Matthew Kennedy, as a consequence of Lemma \ref{kennedy}.

\begin{thm}
    Let $E$ be an operator space, and let $E\subseteq F$ be an injective envelope. Then $E\subseteq F$ is \emph{rigid}, that is, the only completely contractive map $\theta: F\to F$ with $\theta(x)=x$ for all $x\in E$ is the identity map.
\end{thm}

\begin{proof}
    Let $F\subseteq \cc B(H)$. Let $\cc S$ be the compact affine right topological semigroup of all completely contractive self-maps of $\cc B(H)$ that pointwise fix $E$. We have that $F = \vp(\cc B(H))$ for some minimal idempotent $\vp\in \cc S$. By Wittstock's extension theorem, $\theta$ extends to a completely contractive map $\theta'\in \cc S$. Since $\theta\vp = \theta'\vp$ is in the minimal right ideal of $\cc S$ generated by $\vp$, we have that $\vp\theta\vp = \vp^2 = \vp$ by Lemma \ref{kennedy}, thus $\theta$ must fix $F$ pointwise.
\end{proof}

We discuss a slight modification which seems to have connections with the theory of noncommutative Poisson boundaries: \cite{arveson, dp, izumi}. Let $\cc M$ be a von Neumann algebra. For $\vp\in \cc{CC}(\cc M)$ we let $F_\vp := \{x\in\cc M: \vp(x) = x\}$ denote the norm-closed subspace of fixed points of $\vp$. Suppose $\vp$ is continuous in the relative weak* topology (that is, the ultraweak topology) as a map of the unit ball of $\cc B(H)$ to itself (that is, $\vp$ is normal) and that $E\subseteq F_\vp$ is a norm-closed subspace. We consider 
\[\cc T_{E,\vp} := \{\theta\in\cc{CC}(\cc M) : \vp\theta = \theta,\ E\subseteq F_\theta\}.\] 
It is straightforward to check that $\cc T_{E,\vp}$ is an affine subsemigroup of $\cc{CC}(\cc M)$ which is closed in the pointwise-ultraweak topology, with closure being where the normality of $\vp$ is necessary. Moreover, $\cc T_{E,\vp}$ is nonempty as any pointwise-ultraweak cluster point of the sequence 
\[\tau_N := \frac{1}{N}\sum_{k=1}^N \vp^k\] belongs to it. (Any such cluster point may even be seen to be idempotent, for instance, by \cite[Proposition 5.2]{arveson}.) Notice that if $\theta(x) = x$, then $\vp(x) = \vp\theta(x) = \theta(x) = x$, so $F_\theta\subseteq F_\vp$. Thus for any idempotent $e\in\cc I(\cc T_{E,\vp})$, we have that the range of $e$ is contained in $F_\vp$.  If $\vp\in \cc{CC}(\cc B(H))$, then the foregoing applies to $\vp^{**}\in \cc{CC}(\cc B(H)^{**})$ and $E = F_\vp\subseteq (F_{\vp})^{**}\subseteq F_{\vp^{**}}$. Note, however, that $\cc B(H)^{**}$ is not injective, as $\cc B(H)$ isn't nuclear \cite{choi}. 

\begin{cor}
    Let $\cc M$ be an injective von Neumann algebra, $\vp: \cc M\to\cc M$ a completely contractive normal map, and $E\subseteq F_\vp$ a norm-closed subspace. If $e\in T_{E,\vp}$ is a minimal idempotent with range $F$, then $E\subseteq F$ is a rigid inclusion of operator spaces.
\end{cor}

\begin{proof}
     Let $\theta:F\to F$ be a completely contractive map which pointwise fixes $E$. The proof follows the same lines as the previous theorem, noting that if $\theta':\cc M\to \cc M$ is an extension of $\theta$, we may replace it with an extension $\theta''$ satisfying $\vp\theta'' = \theta''$ by taking $\theta''$ to be a pointwise-ultraweak cluster point of $\tau_N \theta'$.
\end{proof}

These arguments can be adapted to a wide variety of related categories: see \cite{cecco} for a detailed treatment of categorical considerations. We outline a few of these.

\begin{enumerate}
    \item If $X\subset \cc B(H)$ is a weakly closed injective subspace, then the set $\cc {CC}(X)$ of completely contractive maps $\phi: X\to X$ is a closed affine right topological subsemigroup of $\cc C(\cc B(H))$.

    \item  If there is a $G$-action by complete isometries on an operator space $E$ and $E\subset X\subset \cc B(H)$ is weakly closed, injective, and $G$-invariant, then the set of $G$-equivariant maps in $\cc{CC}(X)$ which pointwise fix $E$ is easily seen to be a closed subsemigroup of $\cc S$, and the proof Theorem \ref{thm:main} shows the existence of a (relative) $G$-injective envelope for $E$. See \cite{kk, kklru} for more on this and applications to the theory of reduced group C$^*$-algebras and C$^*$-algebras of groupoids.

    \item If $\pi: G\to \cc U(H)$ is a unitary representation, then there is a minimal unital completely positive $\pi$-invariant projection $E: \cc B(H)\to \cc B(H)$. In general there should be many such projections. If $\pi$ is the left regular representation of $G$ on $\ell^2(G)$, then the image of one such projection lies in $\ell^\infty(G)$, thus corresponds with the Furstenberg--Hamana boundary. If $\pi$ is amenable in the sense of Bekka \cite{be}, it is trivial to see that $\bb C 1_{\cc B(H)}$ is one such subspace.
\end{enumerate}

\section*{Acknowledgements}

 The author thanks Mehrdad Kalantar for thoughtful comments and for pointing out Paulsen's work \cite{pa1}. 
 This note is a reworked version of an informal note \cite{sin} which the author publicly posted to his research webpage in October 2015. 
 The author is thankful to Adam Dor-On and Matthew Kennedy for the encouragement to turn that note into a formal manuscript and to the anonymous reviewer for helpful comments. The author was partially supported by NSF grant DMS-2055155.


\begin{bibdiv}
    \begin{biblist}

\bib{at}{book}{
   author={Argyros, Spiros A.},
   author={Todorcevic, Stevo},
   title={Ramsey methods in analysis},
   series={Advanced Courses in Mathematics. CRM Barcelona},
   publisher={Birkh\"auser Verlag, Basel},
   date={2005},
   pages={viii+257},
   isbn={978-3-7643-7264-4},
   isbn={3-7643-7264-8},
   review={\MR{2145246}},
}

\bib{arveson}{article}{
   author={Arveson, William},
   title={The asymptotic lift of a completely positive map},
   journal={J. Funct. Anal.},
   volume={248},
   date={2007},
   number={1},
   pages={202--224},
   issn={0022-1236},
   review={\MR{2329688}},
   doi={10.1016/j.jfa.2006.11.014},
}

\bib{be}{article}{
   author={Bekka, Mohammed E. B.},
   title={Amenable unitary representations of locally compact groups},
   journal={Invent. Math.},
   volume={100},
   date={1990},
   number={2},
   pages={383--401},
   issn={0020-9910},
   review={\MR{1047140}},
   doi={10.1007/BF01231192},
}

\bib{berglund}{book}{
   author={Berglund, J. F.},
   author={Junghenn, H. D.},
   author={Milnes, P.},
   title={Compact right topological semigroups and generalizations of almost
   periodicity},
   series={Lecture Notes in Mathematics},
   volume={663},
   publisher={Springer, Berlin},
   date={1978},
   pages={x+243},
   isbn={3-540-08919-5},
   review={\MR{0513591}},
}

\bib{cecco}{article}{
   author={Cecco, Arianna},
   title={A categorical approach to injective envelopes},
   journal={Ann. Funct. Anal.},
   volume={15},
   date={2024},
   number={3},
   pages={Paper No. 49, 28},
   issn={2639-7390},
   review={\MR{4736304}},
   doi={10.1007/s43034-024-00350-z},
}

\bib{choi}{article}{
   author={Choi, Man Duen},
   author={Effros, Edward G.},
   title={Nuclear $C\sp*$-algebras and injectivity: the general case},
   journal={Indiana Univ. Math. J.},
   volume={26},
   date={1977},
   number={3},
   pages={443--446},
   issn={0022-2518},
   review={\MR{0430794}},
   doi={10.1512/iumj.1977.26.26034},
}

\bib{conway}{book}{
   author={Conway, John B.},
   title={A course in functional analysis},
   series={Graduate Texts in Mathematics},
   volume={96},
   edition={2},
   publisher={Springer-Verlag, New York},
   date={1990},
   pages={xvi+399},
   isbn={0-387-97245-5},
   review={\MR{1070713}},
}

\bib{dp}{article}{
   author={Das, Sayan},
   author={Peterson, Jesse},
   title={Poisson boundaries of ${\rm II}_1$ factors},
   journal={Compos. Math.},
   volume={158},
   date={2022},
   number={8},
   pages={1746--1776},
   issn={0010-437X},
   review={\MR{4493239}},
   doi={10.1112/S0010437X22007539},
}

\bib{dgl}{book}{
   author={Di Nasso, Mauro},
   author={Goldbring, Isaac},
   author={Lupini, Martino},
   title={Nonstandard methods in Ramsey theory and combinatorial number
   theory},
   series={Lecture Notes in Mathematics},
   volume={2239},
   publisher={Springer, Cham},
   date={2019},
   pages={xvi+204},
   isbn={978-3-030-17955-7},
   isbn={978-3-030-17956-4},
   review={\MR{3931702}},
   doi={10.1007/978-3-030-17956-4},
}

\bib{el}{article}{
   author={Ellis, Robert},
   title={Distal transformation groups},
   journal={Pacific J. Math.},
   volume={8},
   date={1958},
   pages={401--405},
   issn={0030-8730},
   review={\MR{0101283}},
}

\bib{fk}{article}{
   author={Furstenberg, H.},
   author={Katznelson, Y.},
   title={Idempotents in compact semigroups and Ramsey theory},
   journal={Israel J. Math.},
   volume={68},
   date={1989},
   number={3},
   pages={257--270},
   issn={0021-2172},
   review={\MR{1039473}},
   doi={10.1007/BF02764984},
}

\bib{gl}{book}{
   author={Glasner, Eli},
   title={Ergodic theory via joinings},
   series={Mathematical Surveys and Monographs},
   volume={101},
   publisher={American Mathematical Society, Providence, RI},
   date={2003},
   pages={xii+384},
   isbn={0-8218-3372-3},
   review={\MR{1958753}},
   doi={10.1090/surv/101},
}

\bib{hp}{article}{
   author={Hadwin, Don},
   author={Paulsen, Vern I.},
   title={Injectivity and projectivity in analysis and topology},
   journal={Sci. China Math.},
   volume={54},
   date={2011},
   number={11},
   pages={2347--2359},
   issn={1674-7283},
   review={\MR{2859698}},
   doi={10.1007/s11425-011-4285-7},
}

\bib{ha}{article}{
   author={Hamana, Masamichi},
   title={Injective envelopes of $C\sp{\ast} $-algebras},
   journal={J. Math. Soc. Japan},
   volume={31},
   date={1979},
   number={1},
   pages={181--197},
   issn={0025-5645},
   review={\MR{0519044}},
   doi={10.2969/jmsj/03110181},
}

\bib{ha1}{article}{
   author={Hamana, Masamichi},
   title={Injective envelopes of operator systems},
   journal={Publ. Res. Inst. Math. Sci.},
   volume={15},
   date={1979},
   number={3},
   pages={773--785},
   issn={0034-5318},
   review={\MR{0566081}},
   doi={10.2977/prims/1195187876},
}

\bib{ha2}{article}{
   author={Hamana, Masamichi},
   title={Injective envelopes of $C^\ast$-dynamical systems},
   journal={Tohoku Math. J. (2)},
   volume={37},
   date={1985},
   number={4},
   pages={463--487},
   issn={0040-8735},
   review={\MR{0814075}},
   doi={10.2748/tmj/1178228589},
}

\bib{ha3}{article}{
   author={Hamana, Masamichi},
   title={Injective envelopes of dynamical systems},
   journal={Toyama Math. J.},
   volume={34},
   date={2011},
   pages={23--86},
   issn={1880-6015},
   review={\MR{2985658}},
}

\bib{hs}{book}{
   author={Hindman, Neil},
   author={Strauss, Dona},
   title={Algebra in the Stone-\v Cech compactification},
   series={De Gruyter Textbook},
   edition={extended edition},
   note={Theory and applications},
   publisher={Walter de Gruyter \& Co., Berlin},
   date={2012},
   pages={xviii+591},
   isbn={978-3-11-025623-9},
   review={\MR{2893605}},
}

\bib{izumi}{article}{
   author={Izumi, Masaki},
   title={$E_0$-semigroups: around and beyond Arveson's work},
   journal={J. Operator Theory},
   volume={68},
   date={2012},
   number={2},
   pages={335--363},
   issn={0379-4024},
   review={\MR{2995726}},
}

\bib{kk}{article}{
   author={Kalantar, Mehrdad},
   author={Kennedy, Matthew},
   title={Boundaries of reduced $C^*$-algebras of discrete groups},
   journal={J. Reine Angew. Math.},
   volume={727},
   date={2017},
   pages={247--267},
   issn={0075-4102},
   review={\MR{3652252}},
   doi={10.1515/crelle-2014-0111},
}


\bib{kklru}{article}{
    author={Kennedy, Matthew},
    author={Kim, Se-Jin},
    author={Li, Xin},
    author={Raum, Sven},
    author={Ursu, Dan},
    title={The ideal intersection property for essential groupoid C*-algebras},
    journal={arXiv e-prints},
    year={2021},
    doi={10.48550/arXiv.2107.03980}  
}

\bib{marrakchi}{article}{
   author={Marrakchi, Amine},
   title={On the weak relative Dixmier property},
   journal={Proc. Lond. Math. Soc. (3)},
   volume={122},
   date={2021},
   number={1},
   pages={118--123},
   issn={0024-6115},
   review={\MR{4210259}},
   doi={10.1112/plms.12347},
}



\bib{pa}{book}{
   author={Paulsen, Vern},
   title={Completely bounded maps and operator algebras},
   series={Cambridge Studies in Advanced Mathematics},
   volume={78},
   publisher={Cambridge University Press, Cambridge},
   date={2002},
   pages={xii+300},
   isbn={0-521-81669-6},
   review={\MR{1976867}},
}

\bib{pa3}{article}{
   author={Paulsen, Vern I.},
   title={A dynamical systems approach to the Kadison-Singer problem},
   journal={J. Funct. Anal.},
   volume={255},
   date={2008},
   number={1},
   pages={120--132},
   issn={0022-1236},
   review={\MR{2417811}},
   doi={10.1016/j.jfa.2008.04.006},
}

\bib{pa1}{article}{
   author={Paulsen, Vern I.},
   title={Weak expectations and the injective envelope},
   journal={Trans. Amer. Math. Soc.},
   volume={363},
   date={2011},
   number={9},
   pages={4735--4755},
   issn={0002-9947},
   review={\MR{2806689}},
   doi={10.1090/S0002-9947-2011-05203-7},
}

\bib{ru}{article}{
   author={Ruan, Zhong-Jin},
   title={Injectivity of operator spaces},
   journal={Trans. Amer. Math. Soc.},
   volume={315},
   date={1989},
   number={1},
   pages={89--104},
   issn={0002-9947},
   review={\MR{0929239}},
   doi={10.2307/2001374},
}

\bib{ryan}{book}{
   author={Ryan, Raymond A.},
   title={Introduction to tensor products of Banach spaces},
   series={Springer Monographs in Mathematics},
   publisher={Springer-Verlag London, Ltd., London},
   date={2002},
   pages={xiv+225},
   isbn={1-85233-437-1},
   review={\MR{1888309}},
   doi={10.1007/978-1-4471-3903-4},
}

\bib{sin}{article}{
   author={Sinclair, Thomas},
   title={A very short proof of the existence of an injective envelope for an operator space},
   pages={\url{https://www.math.purdue.edu/~tsincla/injective-note-1.pdf}}
}


    \end{biblist}
\end{bibdiv}
\end{document}